[2013/29/03]
\documentclass[12pt]{amsart}
\usepackage{bbm}
\chardef\bslash=`\\ 

\newtheorem{thm}{Theorem}[section]
\newtheorem*{thm*}{Theorem}

\newtheorem{lem}[thm]{Lemma}
\newtheorem{prop}[thm]{Proposition}

\theoremstyle{definition}
\newtheorem{defn}{Definition}[section]
\newtheorem{rem}{Remark}[section]
\newtheorem{emp}{Example}[section]
\newtheorem*{notation}{Notation}

\theoremstyle{remark}

\newcommand{\N}{\mathbb{N}}
\newcommand{\R}{\mathbb{R}}

\newcommand{\1}{\mathbbm{1}}

\DeclareMathOperator{\codim}{codim}

\newcommand{\eval}[2][\right]{\relax
  \ifx#1\right\relax \left.\fi#2#1\rvert}

\begin{document}
\title[auerbach]{On explicit constructions of Auerbach bases in separable Banach spaces.}
\author[Robert Bogucki]{Robert Bogucki}
\address{Institute of Mathematics \\
University of Warsaw \\
Banacha 2, 02-097 Warszawa, Poland}
\email{r.bogucki@students.mimuw.edu.pl} 

\begin{abstract}
This paper considers explicit constructions of Auerbach bases in separable Banach spaces. Answering the question of A. Pe{\l}czy{\'n}ski, we prove by construction the existence of Auerbach basis in arbitrary subspace of $c_0$ of finite codimension and in the space $C(K)$ for $K$ compact countable metric space.
\end{abstract}
\maketitle

\section{Introduction}

The problem of Auerbach basis in finite dimensional Banach spaces has a simple, classical solution. Yet, the infinite-dimensional, separable case seems to be much more challenging. By using naive methods, namely Gramm-Schmidt orthogonalisation techniques one can construct a total and fundamental biorthogonal system, however we have no warranty about its boundedness. In order to obtain the boundedness, one needs a more subtle approach. This problem was first solved by A. Pe{\l}czy{\'n}ski and R. Ovsepian in 1975 \cite{PO}. One year later, Pe{\l}czy{\'n}ski \cite{P} strengthened the result by obtaining a "$1+ \epsilon$ Auerbach basis", one of the main ingredients of the proof was the celebrated Dvoretzky theorem. The problem whether every separable Banach space contains an Auerbach basis seems to be still open. A. Pe{\l}czy{\'n}ski \cite{PW} asked about the explicit construction of Auerbach bases in the subspaces of $c_0$ and general $C(K)$ spaces for $K$ compact, countable metric space. In section $2$ we prove the $c_0$ case assuming that the space is of finite codimenision and in section $3$ we present the proof for $C(K)$.

Let $X$ be a separable Banach space. We shall recall some essential definitions.

\begin{notation}
To avoid confusion, by $\left| \cdot \right|$ we will denote the Euclidean norm, whereas $\| \cdot \|$ will always refer to the underlying Banach space norm. By $\left[x_{n}\right]_{n=1}^{\infty}$ we will denote the closed linear span of $\left\{x_{n}\right\}_{n=1}^{\infty}$. 
\end{notation}

\begin{defn}
A sequence $\left(x_{n};x_{n}^{*}\right)_{n=1}^{\infty}$, where $x_{n} \in X$, and $x_{n}^{*} \in X^{*}$ is called:
\begin{enumerate}
\item biorthogonal, if $x_{k}^{*}(x_{j})=\delta_{k,j}$ for $k,j = 1,2,...$,
\item total, if $x_{n}^{*}(x)=0$ for every $n=1,2,...$ implies $x=0$,
\item fundamental, if $\left[x_{n}\right]_{n=1}^{\infty}=X$ (equivalently, $x^{*}(x_{n})=0$ for every $n=1,2,...$ implies $x^{*}=0$),
\item bounded by $M$, if $\|x_{n}\|\|x_{n}^{*}\| \le M$ for every $n=1,2,...$.
\end{enumerate}
\end{defn}

\begin{defn}
Auerbach basis of $X$ is a sequence satisfying (1), (2), (3), with the property that $\|x_{n}\|=\|x_{n}^*\|=1$ for every $n=1,2,...$.
\end{defn}
A sequence satisfying (1), (2), (3) is sometimes called a Markushevich basis.

Recall the following classical results.

\begin{thm*}[Auerbach, \cite{L}]
Every finite dimensional Banach space has an Auerbach basis.
\end{thm*}

\begin{thm*}[Pe{\l}czy{\'n}ski, \cite{P}]
Let $X$ be an infinite-dimensional, separable Banach space. Then, for every $\epsilon > 0$ there exists a fundamental, total, biorthogonal sequence $\left(x_{n};x_{n}^{*}\right)_{n=1}^{\infty}$, satisfying $\|x_{n}\|\|x_{n}^{*}\| \le 1 + \epsilon$.
\end{thm*}

\section{Subspaces of $c_0$.}

We will now construct an Auerbach basis in an arbitrary subspace of $c_0$ of finite codimension. Naturally we assume them to be closed (so that they are still Banach spaces).

\begin{thm}
Let $X$ be a subspace of $c_0$ of finite codimension. Then $X$ has an Auerbach basis.
\end{thm}

\begin{proof}
We start with the case when $\codim X = 1$.
Let us take $f \in c_0^* = l_1$ such that $X = \ker f$. We can expand $f$ in the standard basis, namely $f=\sum\limits_{n=1}^{\infty}a_{n}e_{n}^{*}$. Where $e_n^*$ are coordinate functionals. By permuting indices, we can assume without a loss of generality, that $|a_{1}|=\sup_{n}|a_{n}|$ (supremum is attained and finite since the sequence converges to $0$). For $n = 1,2,...$ let us take
$$
x_{n}=e_{n+1}-\frac{a_{n+1}}{a_{1}}e_{1},
$$
$$
x_{n}^{*}=e_{n+1}^{*},
$$
it is clear that $x_{i}^{*}(x_{j})=\delta_{ij}$ and $\left\Vert x_{n}\right\Vert =\left\Vert x_{n}^{*}\right\Vert =1$ for all $n$. To check that $x_n$ spans the whole $X$, consider an arbitrary $y=(y_{1},y_{2},...)\in X$. We claim that $y=\sum\limits_{n=1}^{\infty}y_{n+1}x_{n}$.
The only non-trivial part is the equality on the first coordinate, however $f(y)=0$, so
$$y_{1}=\frac{-1}{a_{1}}\sum\limits_{n=2}^{\infty}y_{n}a_{n}$$
where the series is absolutely convergent because $\left(a_n\right) \in l_1$. Therefore
$$
\sum_{n=1}^{\infty}y_{n+1}x_{n}=\sum_{n=1}^{\infty}y_{n+1}e_{n+1}-e_{1}\frac{1}{a_{1}}\sum_{n=2}^{\infty}y_{n}a_{n}=\sum_{n=1}^{\infty}y_{n}e_{n}=y.
$$
Totality is also trivially satisfied. Let $y=(y_{1},y_{2},...)\in X$ and suppose that $x_{n}^{*}(y)=0$ for all $n$. This implies $y_{i}=0$ for $i>1$, therefore $f(y)=a_{1}y_{1}=0$ and finally $y_{1}=0$.

We will now generalise this approach to the case when $\codim X = n>1$.
Let
$$
X=\bigcap_{k=1}^{n}\ker f^{k},
$$
$$
f^{i}=\sum_{n=1}^{\infty}f_{n}^{i}e_{n}^{*},
$$
where $f^{j} \in l_{1}$ for $j=1,2,..,n$ are linearly independent functionals on $c_0$. We can now consider an infinite matrix 
$D=[f_{j}^{i}]_{1\leq j\leq\infty}^{1\leq i\leq n}$, the rows of which corresponds to the aforementioned functionals.
Let $A=[f_{j}^{i}]_{1\leq j\leq n}^{1\leq i\leq n}$ be a $n \times n$ matrix consisting of the first $n$ columns of $D$.
Suppose for a moment that $A$ maximizes the absolute value of determinant over all $n \times n$ matrices composed from the columns of $D$. In particular $\det A \neq 0$. We are now in position to construct the biorthogonal sequence. The idea is similar to the case $n=1$, namely we would like to have
$$
x_{k}^{*}=e_{n+k}^{*},
$$
$$
x_{k}=e_{n+k} + v(k),
$$
for $k=1,2,...$, where $v(k) \in \mathrm{span} (e_1,e_2,...,e_n)$. Fix $k$ for a moment and suppose $v(k)=(v_{1},v_{2},...,v_{n},0,...)$. The coordinates can now be computed as follows. The condition $x_k \in X$ implies that that $f^j(x_k)=0$ for $j=1,2,...,n$, thus we have $n$ equations of the form
$$
f_{1}^{j}v_{1}+f_{2}^{j}v_{2}+...+f_{n}^{j}v_{n} + f_{n+k}^{j} = 0.
$$
Rewriting these equations in terms of the matrix $A$, we have
$$
A
\begin{bmatrix}v_{1}\\
..\\
..\\
v_{n}
\end{bmatrix}
=
\begin{bmatrix}-f_{k+n}^{1}\\
..\\
..\\
-f_{k+n}^{n}
\end{bmatrix}.
$$
By using Cramer's formula we obtain
$$
v_{i}=\frac{\det A_{i}}{\det A},
$$
where $A_{i}$ is the matrix $A$ with the $i$-th column substituted for
$$
\begin{bmatrix}
-f_{n+k}^{1}\\
-f_{n+k}^{2}\\
..\\
-f_{n+k}^{n}
\end{bmatrix}.
$$
Notice, that $- \det A_{i}$ is the determinant of a matrix consisting of $n$ first columns of $D$ with $i$-th column substituted for the column with number $n+k$ in $D$. We have assumed that $A$ maximizes the absolute value of determinant over matrices consisting of columns of $D$. Therefore we have
$$
\left|v_{i}\right|=\left|\frac{\det A_{i} }{\det A}\right|\leq1,
$$
which implies that $\left\Vert x_{k}\right\Vert  =1$. The construction itself also ensures that $\left\Vert x_{k}^{*}\right\Vert =1$ and $x_{k}^{*}(x_{j})=\delta_{k,j}$. Note that for every $i$, the sequence $\{ x_k(i) \}_{k=1}^\infty$ is in $l_1$. Indeed, there is nothing to do if $n < i$, otherwise, let us take
$$
M=\sup_{k}\left| (f_{k}^{1},...,f_{k}^{n})\right|.
$$
$M$ is finite since $f^j \in l_1$ for every $1 \le j \le n$.
Using Hadamard inequality yields
$$
\left| x_{k}(i) \right|=\left|\frac{\det A_{i} }{\det A}\right|\leq \frac{M^{n-1}}{\left|\det A \right|} \left| (f_{n+k}^{1},...,f_{n+k}^{n}) \right|
\le \frac{M^{n-1}}{\left|\det A \right|} \left( \left|f_{n+k}^{1}\right| + ... + \left|f_{n+k}^{n}\right| \right).
$$

Let us check that $\{x_{k}\}_{k=1}^{\infty}$ spans the whole X. Consider an arbitrary $y=(y_{1},y_{2},...)\in X$ and set $\tilde{y} = \sum\limits_{k=1}^{\infty}y_{n+k}x_{k}$. This series makes sense because its coordinates are in $l_1$. Once again we claim that $y=\tilde{y}=\sum\limits_{k=1}^{\infty}y_{n+k}x_{k}$. It is clear that we have $y_k = \tilde{y}_k$ for $k>n$. Let us now justify the equality on the first $n$ coordinates. Since $y, \tilde{y} \in X$, we have $f^{j}(y)=f^{j}(\tilde{y})=0$, that is
$$
f_{1}^{j}y_{1}+f_{2}^{j}y_{2}+...+f_{n}^{j}y_{n} = -\sum\limits_{k=n+1}^{\infty}f_{k}^{j}y_{k},
$$
and
$$
f_{1}^{j}\tilde{y}_{1}+f_{2}^{j}\tilde{y}_{2}+...+f_{n}^{j}\tilde{y}_{n} = -\sum\limits_{k=n+1}^{\infty}f_{k}^{j}\tilde{y}_{k} = -\sum\limits_{k=n+1}^{\infty}f_{k}^{j}y_{k},
$$
for every $1\leq j\leq n$ and therefore the vectors $(y_1,y_2,...,y_n)$ and $(\tilde{y}_1,\tilde{y}_2,...,\tilde{y}_n)$ both satisfy the following system of equations
$$
Ax
=
\begin{bmatrix}
-\sum\limits_{k=n+1}^{\infty}f_{k}^{1}y_{k}\\
..\\
..\\
-\sum\limits_{k=n+1}^{\infty}f_{k}^{n}y_{k}
\end{bmatrix}.
$$
Since $A$ is nonsingular, this system has unique solution and the argument follows.

To check that $\{x_k^*\}$ is total, once more suppose $y=(y_{1},y_{2},...)\in X$ and $x_{n}^{*}(y)=0$ for all $k$. This implies $y_{i}=0$ for $i>n$. Once more we can construct a system of equations which determines $(y_1,y_2,...,y_n)$ and $y=0$ follows from the fact that $A$ is nonsingular.

To finish the proof, we need to justify our assumptions on the determinant of $A$. We will say that the matrix is maximal if it maximizes the absolute value of determinant among all $n \times n$ matrices composed from columns of $D$. From the independence of $\{f^{j}\}_{1 \le j \le n}$ we can find $n$ linearly independent columns of $D$. Let us denote by $B$ the matrix, consisting of this columns. There exists $m$ such that $B$ does not have columns with indices greater than $m$. Once again, recall that
$$
M=\sup_{k}\left| (f_{k}^{1},...,f_{k}^{n})\right|,
$$
is finite, since $f^{j} \in l_{1}$ for $j=1,2,..,n$. Moreover, we can select $N > m$ so large, that the following inequality will hold for any $v$ column of $D$ with number greater that $N$
$$
\left| v\right| < \frac{\left|\det B \right|}{M^{n-1}}.
$$
Once again recall Hadamard inequality, which implies that a maximal matrix cannot have columns with numbers greater than $N$. Therefore, to find a maximal matrix, one has to consider only the first $N$ columns of $D$. By permuting indices we could have assumed that the first $n$ columns of $D$ were maximal.
\end{proof}

\begin{rem}
The assumption that the codimension of $X$ is finite was crucial. In fact, the author does not know the answer in the case where both, dimension and codimension are infinite.
\end{rem}

\section{Continuous functions on a countable metric compact}

In this section, we start with an elementary construction in the space $c$. This simple idea will then be extended and the proof of the general case will follow from transfinite induction and a suitable isometric classification of $C(K)$ spaces.
\begin{prop}\label{cc}
Let $K$ be a countable metric compact with finitely many accumulation points. Then $C(K)$ has an Auerbach basis.
\end{prop}

\begin{proof}
Let $n$ be the number of accumulation points. We will first consider the case when $n=1$ (so we are looking at the space $c$).
Assume that $K = \{z_k\}_{k=1}^{\infty} \cup \{ z \}$ where $z_k \rightarrow z$. Define $Z_j = \{z_k\}_{k=j+1}^{\infty} \cup \{z\}$. Now take $x_0 \equiv 1$, $x_0^* = \sum_{k=1}^{\infty} \delta_{z_k} 2^{-k}$, and for $j=1,2,...$ 
$$
x_j  =  - \1_{z_j} + \1_{Z_j},
$$
$$
x_j^* = \sum_{k=j}^{\infty} x_j(z_k) \delta_{z_k} 2^{j-k-1} = -\frac{\delta_{z_j}}{2} + \sum_{k=j+1}^{\infty} \delta_{z_k} 2^{j-k-1}.
$$
It is clear that $\| x_j \| = \|x_j^*\| = 1$ and the sequence $\left( x_j ; x_j^* \right)_{j=0}^{\infty}$ is biorthogonal. Fundamentality becomes clear when one realizes that the matrix
$$
\left( \begin{array}{cccccc}
1 & 1 & 1 & ... & 1 \\
-1 & 1 & 1 & ... & 1 \\
0  & -1 & 1 & ... & 1 \\
...  &  ... & ... & ... & ... \\
0 & 0 & 0 & -1 & 1 \end{array} \right)
$$
is nonsingular and therefore every $x \in C(K)$ with finite amount of values can be expressed as a linear combination of $\{x_j\}_{j=1}^{\infty}$. In order to check totality, take $y \in X$ such that $x_j^*(y) = 0$ for $j=0,1,2,...$. To simplify notation, set $y_j = y(z_j)$. We now arrive at the following system of equations
$$
\sum_{k=1}^{\infty} y_k 2^{-k} = 0
$$
and
$$
-y_j + \sum_{k=j+1}^{\infty} y_k 2^{j-k-1} =0
$$
for $j=1,2,...$. It is elementary to check that this forces $y_j=0$ and consequently $y \equiv 0$.

For the case $n>1$ split $K$ into disjoint copies $K_1,K_2,...,K_n$ where each $K_j$ is just as in the case $n=1$ and apply the same construction.
\end{proof}

\begin{prop}
Let $X$ be a Banach space with Auerbach basis, then $c(X)$ has an Auerbach basis.
\end{prop}

\begin{proof}
Suppose that $\left( x_m ; x_m^* \right)_{m=1}^{\infty}$ forms an Auerbach basis in $X$. Set 
$$
v_0 = \sum_{k=1}^{\infty} e_k,
$$
$$
v_0^* =  \sum_{k=1}^{\infty} e_k^* 2^{-k}.
$$
and consider the following sequences
$$
v_n= - e_k + \sum_{k=n+1}^{\infty} e_k,
$$
$$
v_n^* = \sum_{k=1}^{\infty} v_n(k) e_k^* 2^{n-k-1} = -\frac{e_n^*}{2} +  \sum_{k=n+1}^{\infty} e_k^* 2^{n-k-1} .
$$
where $n=1,2,...$.
Now we can define
$$
x_{n,m} = v_n \cdot x_m,
$$
$$
x_{n,m}^* = v_n^* \cdot x_m^*,
$$
where $\cdot$ is pointwise product and $n,m = 1,2,...$. It is obvious that $\| x_{n,m} \| = \| x_{n,m}^* \| = 1$ and this sequence is biorthogonal. To check that it is fundamental, consider $\epsilon > 0$ and an arbitrary sequence $\left(y_1,y_2,...\right) = y \in c(X)$. For a sufficiently large index $N$, we have
$$ \| y - (y_1,y_2,...,y_N,y_N,...)\| < \epsilon /2.$$
Now, for each $y_j$, where $j=1,2,..,N$ we can find $\alpha_j \in c_{00}$, such that 
$$
\| y_j - \sum_{k=1}^{\infty} \alpha_j(k) x_k\| < \epsilon/2.
$$
Set $\alpha_j(k) = \alpha_N(k)$ for $j > N$, $k=1,2,...$ and let $M$ be an integer guaranteeing that $\alpha_j(k) = 0$ for $j=1,2,...$ whenever $k>M$. Recalling the matrix from Preposition \ref{cc}, one sees that for every $k$, we can find $\left\{\beta_{j,k}\right\}_{j=1}^{\infty} \in c_{00}$, such that
$$
\tilde{x}_k := \sum_{j=1}^{\infty} \beta_{j,k} x_{j,k}
$$
satisfies $\tilde{x}_k (j) =  \alpha_j (k)$ for $j=1,2,...$
Take 
$$
x = \sum_{k=1}^{M} \tilde{x}_k = \sum_{k=1}^{M} \sum_{j=1}^{\infty} \beta_{j,k} x_{j,k}.
$$ 
From the construction, we see that $x$ lies in the span of $\left\{x_{n,m}\right\}_{n,m=1}^{\infty}$ and $\| y - x \| < \epsilon$. In order to prove that the sequence is total, fix $m$. Arguing similarly as in Prepositon \ref{cc}, one sees that if $x_{n,m}^* (y) = 0$ for every $n = 1,2,...$, then $x_m^*(y_k) = 0$ for every $k=1,2,...$. Since $\left\{x_m^* \right\}_{m=1}^{\infty}$ was total, the claim follows.
\end{proof}

Before going any further, let us recall some basic facts about $C(K)$ spaces, where $K$ is countable compact metric space. For a complete survey, see \cite{ROS}.

\begin{thm*}
For every $K$ infinite countable compact metric space, there exists a unique ordinal $1 \le \alpha < \omega_1$ and a unique $n \in \N$ such that the space $C(K)$ is isometric to $\underbrace{C\left(\omega^\alpha+\right) \oplus ... \oplus C\left(\omega^\alpha+\right)}_\text{$n$}$ where the direct sum is taken in the supremum norm.
\end{thm*}

Let us also recall the following construction from \cite{ROS}.
Set $X = \oplus_{n=1}^{\infty} X_n$ where each $X_n = C(K_n)$ for some compact metric space $K_n$. By $c_0 (X) \oplus 1$ we will denote the "unitization" of $c_0 (X)$. That means we consider pairs $\left(\{x_n\}_{n=1}^\infty, c\right) \in X \times \R$ with the norm 
$$\left\| \left(\{x_n\}_{n=1}^\infty,c\right) \right\| = \sup_n \sup_{\omega \in K_n} | x_n(\omega) + c |.$$

Now define $Y_\alpha$ as follows. Let $Y_1 = c_0 \oplus 1 = c$. Now, we can proceed by induction. If $\beta = \alpha + 1$ set $Y_\beta = c_0(Y_\alpha) \oplus 1$ and if $\beta$ is a limit ordinal, choose $\alpha_n \nearrow \beta$ and set $Y_\beta = \left(Y_{\alpha_1} \oplus Y_{\alpha_2} \oplus ... \right)_{c_0} \oplus 1$. 

It turns out that these spaces can be seen as building blocks for $C(K)$ spaces. Namely, we can reformulate the previous theorem as follows (see \cite{ROS}).

\begin{thm*}
For every $K$ infinite countable compact metric space, there exists a unique ordinal $1 \le \alpha < \omega_1$ and a unique $n \in \N$ such that the space $C(K)$ is isometric to $\underbrace{Y_\alpha \oplus ... \oplus Y_\alpha}_\text{$n$}$ where the direct sum is taken in the supremum norm.
\end{thm*}

It is now clear that in order to construct an Auerbach basis in an arbitrary $C(K)$ space for $K$ countable compact metric space, it is enough to consider the case $Y_\alpha$ or $C\left(\omega^\alpha+\right)$. It turns out, that the rudimentary construction from Preposition \ref{cc} can be easily extended to the general case. Before giving the proof, let us provide an illuminating example.
\begin{emp}
Consider $K_j = \left\{ \frac{1}{n} \mid n \ge j \right\} \cup \{ 0 \}$ and set 
$$
K = \left(0,0\right) \cup \bigcup_{j \ge 1} \left\{ \frac{1}{j} \right\} \times K_j.
$$
$K \in \R^2$ with the euclidean distance is a compact metric space and $C(K)$ is isometric to $\left(c \oplus c \oplus ... \right)_{c_0} \oplus 1$ and $C(\omega^2+)$, while $C(K_j)$ is isometric to $c$. In \ref{cc} we have constructed an Auerbach basis $\left(x_k;x_k^*\right)_{k=0}^{\infty}$ of $c$ such that $x_0 \equiv 1$. Therefore we can proceed as follows. For each $K_j$ take the sequence $\left(x_k;x_k^*\right)_{k=0}^{\infty}$ and let $\left(x_{k,j};x_{k,j}^*\right)_{k=0}^{\infty}$ be its extension by zero to the whole $K$. Now once again consider the basis $\left(x_k;x_k^*\right)_{k=0}^{\infty}$ but now treat the sequence $\left\{x_k(j)\right\}_{j=1}^{\infty}$ as a function on $C(K)$ which is constant on each $K_j$ and equal to $x_k(j)$. We can proceed similarly for the functionals and obtain a formula
$$
x_k = \sum_{j=1}^{\infty} x_k(j) \1_{K_j} = \sum_{j=1}^{\infty} x_k(j) x_{0,j},
$$
$$
x_k^* = \sum_{j=1}^{\infty}  2^{-j} x_k^*(j) x_{0,j}^*.
$$
We claim that $\left(x_k;x_k^*\right)_{k=0}^{\infty} \cup \bigcup_{j \ge 1} \left(x_{k,j};x_{k,j}^*\right)_{k=1}^{\infty}$ is an Auerbach basis in $C(K)$ (notice that we omit the functions $x_{0,j}$). Indeed, it is clear that all the vectors and all the functionals have norm $1$. Biorthogonality follows from the properties of underlying basis of $c$. In order to check fundamentality, a moment of reflection should convince us that every function $f \in C(K)$ with finite amount of values can be expressed as a linear combination of vectors from the proposed basis. To see that the basis is total it is enough to understand that if $x_{k,j}^*(f)=0$ for $k=1,2,...$, then $f$ is constant on $K_j$ and if a function is constant on each $K_j$ we may treat it as an element of $c$ and use totality of $\{x_k^*\}_{k=0}^{\infty}$ to conclude that the function is zero on the whole $K$.
\end{emp}

A similar procedure will now allow us to prove the key lemma.

\begin{lem}\label{key}
For $j=1,2,...$ let $K_j$ be a metric space and suppose that $\left(f_{k,j};f_{k,j}^*\right)_{k=0}^{\infty}$ is an Auerbach basis in $C(K_j)$ such that $f_{0,j} \equiv 1$. Then the space 
$$
X := \left(C(K_1) \oplus C(K_2) \oplus ... \right)_{c_0} \oplus 1
$$
has an Auerbach basis $\left(x_{k};x_{k}^*\right)_{k=0}^{\infty}$ such that  $x_0  = \left(\mathbf{0}, 1\right)$.
\end{lem}

\begin{proof}
We start by introducing some notation. Given a function $f \in C(K_j)$, we set 
$$
f e_j := (\underbrace{0,...0}_\text{$j-1$},f,0,...) \in \left(C(K_1) \oplus C(K_2) \oplus ... \right)_{c_0},
$$ similarly for $f^* \in C(K_j)^*$, 
$$
f^* e_j^* := (\underbrace{0,...0}_\text{$j-1$},f^*,0,...) \in \left(C(K_1)^* \oplus C(K_2)^* \oplus ... \right)_{l_1},
$$
and finally 
$$
f^* e_j^* (f e_k) = 
\left\{
\begin{array}{l l}
f^*(f) & \quad \text{if $j=k$,} \\
0 & \quad \text{otherwise.} \\
\end{array}
\right.
$$
For $(x,c) \in X$ we may write 
$$
(x,c) = \sum_{j=1}^{\infty} (x_j + c\1_{K_j}) e_j,
$$
and therefore we may interpret elements of $X$ and $X^*$ as linear combinations of $e_j$ or $e_j^*$ with suitable coefficients.
We now mimic the construction from example in order to construct the basis. For $j=1,2,...$ and $k=1,2,...$ consider
$$
x_{k,j} = f_{k,j} e_j,
$$
$$
x_{k,j}^* = f_{k,j}^* e_j^*.
$$
Notice, that we purposely omit the constant functions. Set $x_0 = (0,1) =  \sum_{j=1}^{\infty} \1_{K_j} e_j. $ and define $x_0^*$ by
$$
x_0^* = \sum_{j=1}^{\infty} 2^{-j} f_{0,j}^* e_j^*,
$$
and for $k=1,2,...$
$$
x_k = -e_k + \sum_{j=k+1}^{\infty} e_j,
$$
$$
x_k^* = \frac{-f_{0,k}^*e_k^*}{2} + \sum_{j=k+1}^{\infty} \frac{f_{0,j}^*e_j^*}{2^{j-k+1}}.
$$
We claim that $\left(x_k;x_k^*\right)_{k=0}^{\infty} \cup \bigcup_{j \ge 1} \left(x_{k,j};x_{k,j}^*\right)_{k=1}^{\infty}$ is an Auerbach basis in $X$. It is clear that $\| x_k \| = \| x_k^* \| = \|x_{k,j} \| = \| x_{k,j}^* \| = 1$. Biorthogonality follows from the fact that $\left(f_{k,j};f_{k,j}^*\right)_{k=0}^{\infty}$ were Auerbach bases containing constant function $1$ and that $\left(x_k;x_k^*\right)_{k=0}^{\infty}$ is essentially the Auerbach basis from $c$ in disguise. As for totality and fundamentality, the double-indexed part of the basis corresponds to setting the value on each block $C(K_j)$ independently, up to the constant factor. On the other hand, the single-indexed part is connected to providing the missing constants and the scalar value. More formally, given $(y,c) \in X$ and $\epsilon > 0$, we find $N$, a sequence $g_j \in C(K_j)$, where $j=1,2,...,N$ and $t \in \R$ such that 
$$
\left\|(y,c) - \left[\sum_{j=1}^{N} g_j e_j + \sum_{j=N+1}^{\infty} t\1_{K_j} e_j \right] \right\| < \epsilon /2.
$$
We can approximate each $g_j$ by $h_j  =  \sum_{k=0}^{\infty} \alpha_{j} (k) f_{k,j}$, where $\alpha_j \in c_{00}$ and $\left\| g_j - h_j \right\| < \epsilon / 2$. It is also possible to find $\beta \in c_{00}$ such that 
$$
\sum_{j=0}^{\infty} \beta_j x_j = \sum_{j=1}^{N} \alpha_{j}(0) \1_{K_j} e_j + \sum_{j=N+1}^{\infty} t\1_{K_j} e_j.
$$
Therefore
\begin{align*}
& \left\| \left[\sum_{j=1}^{N} g_j e_j + \sum_{j=N+1}^{\infty}  t \1_{K_j} e_j \right] - \left[\sum_{j=0}^{\infty} \beta_j x_j + \sum_{j=1}^{N} \sum_{k=1}^{\infty} \alpha_{j} (k) x_{k,j}   \right] \right\| \\
& =  \left\| \sum_{j=1}^{N} g_j e_j - \sum_{j=1}^{N} h_j e_j  \right\| < \epsilon / 2,
\end{align*}
and we have just proven fundamentality. To check that the sequence is total, notice that for $y = \sum_{j=1}^{\infty} y_j e_j$, the condition $x_{k,j}^*(y) = 0$ for $k=1,2,...$ implies that $f_{k,j}^*\left(y_j - f_{0,j}^*(y_j)f_{0,j}\right) = 0$ for $k=0,1,2,...$ and therefore $y_j$ is a constant function. If $y_j$ is constant for every $j=1,2,...$, then $\left(x_k;x_k^*\right)_{k=0}^{\infty}$ mimics an Auerbach basis in $c$ and totality is established.

\end{proof}

\begin{thm}
Let $K$ be a countable compact metric space, then $C(K)$ has an Auerbach basis.
\end{thm}

\begin{proof}
Using transfinite induction and lemma \ref{key} we can find a basis for all $Y_\alpha$ where $1 \le \alpha < \omega_1$. From the aforementioned isometric classification of $C(K)$ spaces, we know that $C(K)$ is isometric to $\underbrace{Y_\alpha \oplus ... \oplus Y_\alpha}_\text{$n$}$ for some $1 \le \alpha < \omega_1$ and $n \in \N$, where the direct sum is taken in the supremum norm. It is now enough to take $n$ distinct copies of Auerbach basis in $Y_\alpha$.
\end{proof}

\section{Acknowledgments}

The author would like to express his sincere gratitude to M. Wojciechowski for encouragement, fruitful discussions and introduction into the topic.
\bibliography{auerbach}
\bibliographystyle{plain}

\end{document}